\documentclass[12pt]{article}
\usepackage{amssymb,amsmath,makeidx,amscd,verbatim,amsthm,graphicx,mathrsfs}
\usepackage{lscape}
\usepackage{bbm}
\usepackage{color}
\usepackage{hyperref}
\usepackage[english]{babel}
\addto{\captionsenglish}{%

}
\newtheorem{thm}{Theorem}[section]

\theoremstyle{definition}
\newtheorem{defx}[thm]{Definition}
\newtheorem{rem}[thm]{Remark}

\numberwithin{equation}{section}

\newcommand{\D}{I\!\!D}

\newcommand{\be}{\begin{enumerate}}
\newcommand{\ee}{\end{enumerate}}
\newcommand{\bq}{\begin{eqnarray*}}
\newcommand{\eq}{\end{eqnarray*}}

\begin{document}
\pagenumbering{arabic} \baselineskip 10pt
\newcommand{\disp}{\displaystyle}
\renewcommand*\contentsname{Table of Contents}
\thispagestyle{empty}
\newcommand{\HRule}{\rule{\linewidth}{0.1mm}}
\linespread{1.0}
\pagenumbering{arabic} \baselineskip 12pt
\thispagestyle{empty}
\title{\textbf{Quaternionic Spherical Fourier Multipliers for Gelfand Pairs}}
\author{O. A. Ariyo$^1$  and  M. E. Egwe$^{2*}$\\ Department of Mathematics, University of Ibadan, Ibadan, Nigeria.\\
$^1$\emph{oa.ariyo@ui.edu.ng}\;\;$^2$\emph{murphy.egwe@ui.edu.ng}}
\maketitle
\large	
\begin{abstract}
Let $\mathbb{H}_q$ denote the quaternionic Heisenberg group with $\mathbb{R}^4\times\mathbb{R}^3$ stratification and $K$ a compact subgroup of Automorphism. We construct the spherical Fourier multiplier related to a Gelfand Pairs on the Quaternionic Heisenberg group $\mathbb{H}_q$ via the motion group $G=K\ltimes \mathbb{H}_q$ of $\mathbb{H}_q$, where $K\ltimes \mathbb{H}_q$ is the semidirect product of $\mathbb{H}_q$ and its compact subgroup of automorphism $K\in Aut(\mathbb{H}_q)$.
\end{abstract}

\small \textbf{Keywords:} Quaternionic Heisenberg Group, Spherical Fourier multiplier, Semidirect Product, Gelfand Pairs.
\small\textbf{Mathematics Subject Classification (2020):}20F18, 22E25, 42A45, 42B15, 43A22\\
\ \\
\section{Introduction}
The construction of spherical Fourier multipliers on Quaternionic Heisenberg group is an interesting study and this can be considered in connection to the Gelfand Pairs defined on the group. The Quaternionic Heisenberg group is an example of Heisenberg group of $H$-type. This group was introduced by A. Kaplan and F. Ricci \cite{kaplan2006harmonic} whose background has provided a framework for several analysis on stratified Lie groups since its inception \cite{chang2006geometric}, \cite{yang2022harmonic}. A quaternion is the quadruple $(1, i,j,k)$ and  its quadrinomial form is written as $q=q_0+iq_1+jq_2+kq_3;~~q\in \mathbb{R}$, where $i^2=j^2=k^2=-1=ijk$. The pure quaternion is the triple $(i,j,k)$. Any unit quaternion is a quaternion with norm $1$ and it's given as $\disp{\hat{q}=\frac{q}{||q||}}$.\\
It is easy to see that $\hat{q}^{-1} =\hat{q}^*$, where $q^*$ denotes the conjugate of $q$. A unit quaternion in matrix form is an orthogonal matrix such that $\disp{A^{-1}=A^*=A^T}$. Any rotation matrix $A$ is an orthogonal matrix with its usual representation.\\
The set of quaternions together with the operations of addition and multiplication satisfies the condition for a ring and hence the triple $(\mathbb{H},+,\cdot)$ is a non-commutative ring since $ij\ne ji, jk\ne kj$ and $ki\ne ik$. The construction of Gelfand Pairs via the motion group $K\ltimes \mathbb{H}_q$ (\cite{faress2020spherical}, \cite{wolf2007harmonic}) of semidirect product of $\mathbb{H}_q$ and $K$ its compact subgroup of automorphism allows the analysis of spherical Fourier multipliers on $\mathbb{H}_q$; i.e. extension of commutative analysis on noncommutative groups.\\
Several analysis on Heisenberg and Quaternionic Heisenberg Groups such as the Gelfand Pairs, spherical functions, spherical-radial multipliers on Heisenberg group and some other harmonic analysis are well known \cite{egwe2014k},\cite{egwe2020spherical},\cite{faress2020spherical}. Here, we consider a construction of spherical Fourier multipliers on the Quaternionic Heisenberg group in the sense of (\cite{egwe2020spherical}, \cite{mensah2024spherical}).\\
In classical abstract harmonic analysis, the theory of Fourier multipliers on a locally compact group $G$ play some crucial roles yielding several results and applications \cite{de1979lp}, \cite{folland2016course}. Multipliers for groups are defined in various ways \cite{larsen2012introduction}, they originated from the study of harmonic analysis of summability of Fourier series of some periodic integrable functions $f\in G$ (\cite{larsen2012introduction}, \cite{ruzhansky2015p}).\\
The concept of multipliers on the Heisenberg group was first introduced by Mauceri et al \cite{mauceri2006harmonic}. Other authors such as F. Ricci and R.L. Rubin \cite{ricci1986transferring},  considered some analysis of Fourier Multipliers and transfer of results from $SU(2)$ onto the Heisenberg group, M.E. Egwe did results on spherical-radial multipliers on the Heisenberg group and some constructions of spherical functions \cite{egwe2020spherical}.\\
Gelfand Pairs on the Heisenberg group $\mathbb{H}_n$, their associated properties and relevance cannot be overemphasized, these have been studied by several authors \cite{gelfand1988collected}, \cite{kaplan2006harmonic} among others. Benson, Jenkins and Ratclif \cite{benson1997geometric} did analysis of Gelfand Pairs on the Heisenberg $\mathbb{H}_n$ of dimension $(2n+1)$. They considered the pair $(K,\mathbb{H}_n)$ as a Gelfand Pair for $K$ a compact subgroup of automorphism $K\subset Aut(\mathbb{H}_n)$. It was clear that $(K, \mathbb{H}_n)$ forms Gelfand Pair if the algebra $L^1_K(\mathbb{H}_n)$ of $K$-invariant integrable functions on $\mathbb{H}_n$ is commutative; that is, the commutativity of the algebra of double coset space $L^1(K\backslash \mathbb{H}_n \ltimes K\slash K)$.
\section{Preliminaries}
The multiplier defined on a Banach algebra is a continuous linear operator $\mathfrak{m}: A\longrightarrow A$ such that $$\mathfrak{m}(x)y=x(\mathfrak{m}y);~~\forall~x,y\in A.$$
\begin{defx}
	Let $G$ be a locally compact abelian group, $L^1(G)\longrightarrow L^1(G)$ be a continuous linear transformation and $M(G)$ denotes the Banach space of bounded regular complex valued Borel measures on $G$ with the norm of $\nu\in M(G)$ equals total variation of the measure $\nu$, i.e. $\disp{\|\nu\|=|\nu|}$. We define a multiplier for $L^1(G)$ as a continuous linear operator $\mathfrak{m}$ on $L^1(G)$ satisfying:
	\begin{enumerate}
		\item[($i$)]$\mathfrak{m}$ commutes with translation operators via $\mathfrak{m}\tau_a=\tau_a\mathfrak{m}$, $\forall a\in G$; where $\tau_a$ is a translation operator such that $(\tau_af)x=f(xa^{-1})$.
		\item[($ii$)] $\mathfrak{m}(f*g)=\mathfrak{m}f*g,~~~\forall~~f,g\in L^1(G)$.
		\item[($iii$)] there exist a unique measure $\nu\in M(G)$. For each $f\in L^1(G)$, we have $$\widehat{\mathfrak{m}f}=\hat{\nu}\hat{f}.$$
		\item [($iv$)]$\exists$ a unique measure $\varphi$ on $\hat{G}$ (the dual space of $G$), such that $$\widehat{\mathfrak{m}f}=\hat{\varphi}\hat{f}.$$
		\item[($v$)] for $\nu \in M(G)$, we have $\mathfrak{m} f=f*\nu,~~\forall~~f\in L^1(G)$.
	\end{enumerate}
\end{defx}
\begin{rem}\ \\
	(1) By $(iii)$ and $(iv)$ $\hat{\varphi} =\hat{\nu}$.\\
	(2) By $(i)$, the multipliers behave in a similar way as the intertwinning operator in classical group representation of the Quaternionic Heisenberg group.\\
	(3) The norm on $L^1(G)$ of equivalence classes of integrable complex-valued functions with Haar measure $\lambda$ on $G$ is given by $$\|f\|_1=\int_G|f(x)|d_\lambda(x);~~f\in L^1(G).$$
\end{rem}
\begin{thm}[\cite{larsen2012introduction}, pg 6]\ \\
Let $G$ be a locally compact abelian group. Then the space of multipliers for $L^1(G)$ is isometric isomorphic to $M(G)$.
\end{thm}
\begin{rem}
This theorem, though defined here for abelian groups $G$, is easily extended to non-commutative groups such as the stratified (Quaternionic) Heisenberg groups since the motion group structure on the quaternionic Heisenberg group $\mathbb{H}_q$ and the Banach algebra $L^1(K\backslash \mathbb{H}_q\ltimes K/K)$ of double coset is commutative.
\end{rem}

\begin{defx}\cite{wolf2007harmonic}
Let $G$ be a locally compact group and $K$ a compact subgroup of $G$. A function $f: G \longrightarrow \mathbb{C}$ is said to be $K$-bi-invariant with respect to $K$ if it is constant on double cosets of $K$, i.e. if $$f(k_1uk_2)=f(u)~~~~~~~~~~\forall k_1,k_2\in K, \forall u\in G.$$
The collection of all $K$-bi-invariant functions on with compact support in $G$ is denoted by $C_c^\#(G)$. Hence, given any $f\in C_c(G)$, the projection $C_c(G)\longrightarrow C_c(K\backslash G/K)$ defined by
\begin{equation}
	f^{\#}(u) = \int_K\int_Kf(kuk^\prime)dkdk^\prime
	\end{equation}
	is in $C_c^{\#}(G)$.
\end{defx}
\begin{defx}
Let $\varphi$ be any complex-valued continuous function on $G=\mathbb{H}_q\ltimes K$, such that $\varphi\ne 0$, then $\varphi$ is spherical if and only if $$\int_{K}\varphi(ukv)dk=\varphi(u)\varphi(v)$$
\end{defx}
\section{The Quaternionic Heisenberg Group $(\mathbb{H}_q)$}
Quaternionic Heisenberg group like its Heisenberg counterpart plays core role in abstract harmonic analysis, the representation theory, analysis of several complex variables, the partial differential equations and quantum mechanics. It is a stratified Lie group with the underlying manifold structure $\mathbb{H}_q=\mathbb{H}\oplus \mathbb{R}^3\approx \mathbb{R}^4\times \mathbb{R}^3$  and multiplication given by \[(u,v)(r,s)=(u+r, v+s-2\Im(\bar{r},u))\]  $\forall~~u,r\in\mathbb{R}^4$ and $v,s\in\mathbb{R}^3$.\\
The centre of quaternionic Heisenberg group $\mathbb{H}_q$ is $\mathbb{R}^3=[\mathbb{H}_q,\mathbb{H}_q]$, and the Haar measure on $\mathbb{H}_q$ is the $dudt, ~\hbox{for}~u\in\mathbb{R}^4 ~\hbox{and} ~t\in\mathbb{R}^3$. The Haar measure on $G=\mathbb{H}_q\ltimes K$ is $dudtdk$ where $dk$ is a Haar measure on $K$.
\begin{defx}
Let $G$ be a locally compact group and $K$ its compact subgroup of automorphism. The semidirect product of $G=\mathbb{H}_q\ltimes K$ on $\mathbb{H}_q$ is defined by \[(k,u,t)\cdot(k^\prime,u^\prime,t^\prime)=(k\cdot k^\prime, (u,t)(k\cdot u^\prime,t^\prime))\] for $k,k^\prime \in K, ~u,u^\prime \in \mathbb{R}^4 ~\hbox{and}~t,t^\prime \in \mathbb{R}^3$.
\end{defx}
Of course, we can show that $K$ is an invariant subgroup of $G=\mathbb{H}_q\ltimes K$, i.e. $kGk^{-1}\in K$ or $KGK^{-1}=K$.\\
We know that every element of $G=\mathbb{H}_q\ltimes K$ can be written as a product of an element of $K$ and $\mathbb{H}_q$ such that $\mathbb{H}_q\cap K=\{e\}$. In fact, the transitive action of $G$ on $\mathbb{H}_q$ given as
\begin{eqnarray*}
&\Psi: G\times \mathbb{H}_q\longrightarrow \mathbb{H}_q\\
&((k,u,t), (v,s))\mapsto (u,t) \cdot (k\cdot v, s),~~~k\in K, ~u \in \mathbb{R}^4 ~\hbox{and}~s\in \mathbb{R}^3;
\end{eqnarray*}

is a transitive action of $G$ on $\mathbb{H}_q$ and identifies the quotient $G\slash K$ with $\mathbb{H}_q$, i.e. $G\slash K \simeq \mathbb{H}_q$ with $K$ the stabilizer of $(0,0)$.\\
To see this, it suffices to show that $xkx^{-1}\in K~\forall ~x\in G, k\in K$. By extension, $xKx^{-1}=K \forall x\in G$.\\
Now $xk\in xK=Kx \implies xk=kx$ for some $k\in K$, then $xkx^{-1}=k\in K$. Now suppose $xk\in xK$, then $xkx^{-1}\in K$, hence $xk=xkx^{-1}x\in Kx$. Thus, $xK\subseteq Kx$. Furthermore $kx=x(x^{-1}kx)\in xK$, showing that $xK=Kx$.\\
Therefore, $xKx^{-1}=K$ because $xkx^{-1}\in K$ and $x^{-1}k(x^{-1})^{-1}\in K, \forall x\in G$ and $k\in K$. This establishes the existence of the semidirect product structure of the motion group $G=\mathbb{H}_q\ltimes K$.\\
The action by automorphism of the group $\mathbb{H}_q$ is described as follows:\\
Let $(u,t)\in\mathbb{H}_q$ and $k\in K$, we define $\varphi_k\in \mathbb{H}_q$ by
\begin{eqnarray*}
	\varphi_k: &\mathbb{H}{q}&\longrightarrow Aut(\mathbb{H}_q)\\	
	&(u,t)&\longmapsto \varphi_k(u,t) =(k\cdot u, t)
\end{eqnarray*}
$\varphi_k$ is a bijection, for if $(u,t), (v,s)\in\mathbb{H}_q,~k\in K$,
\begin{eqnarray*}
\varphi_k(u,t)\cdot\varphi_k(v,s)&=&\left(k\cdot u+k\cdot v, t+s-2\Im(\overline{k\cdot v}\cdot k\cdot u)\right)\\
&=&\left(k\cdot u+k\cdot v, t+s-2\Im(\bar{v}\cdot\bar{k}\cdot k\cdot u)\right)\\
&=&\left(k\cdot (u+v), t+s-2\Im(\bar{v}\cdot u)\right)\\
&=&\varphi_k\left((u,t)\cdot(v,s)\right).
\end{eqnarray*}

\section{Positive Definite Function on Quaternionic Heisenberg Group}
Let $\mathbb{H}_q=\mathbb{H}\oplus \mathbb{R}^3$ be a stratified quaternionic Heisenberg group. A function $\varphi: \mathbb{H}_q\longrightarrow \mathbb{C}$ is positive definite if the inequality
\begin{equation}\label{4.1} \sum_{i=1}^{n}\sum_{j=1}^{n}\varphi\left((u,t)_i(u,t)^{-1}_j\right)z_i\bar{z_j}\ge 0
 \end{equation} holds.\\
We note that $z_i z_j=z_i\bar{z_j},~~iff~~z_i, z_j \in \mathbb{R}$. \eqref{4.1} is true since $\mathbb{H}_q=\mathbb{R}^4\times \mathbb{R}^3,~\forall~n\in \mathbb{N},~(u,t)_1, (u,t)_2,\cdots,(u,t)_i\in\mathbb{H}_q$ and $z_1,\cdots,z_j\in\mathbb{C}$.\\
For a continuous function $\varphi:\mathbb{H}_q\longrightarrow \mathbb{C}$, we define equivalently, the positive definite function on $\mathbb{H}_q$ by;
$$\int\int_{\mathbb{H}_q}f\left((u,t)(v,t)^{-1}\right)\overline{g (u,t)}\cdot g(v,t)dudv\ge 0$$
for each of the functions $f,g\in L^1(\mathbb{H}_q)$.
The following results hold true for positive definiteness:
\begin{enumerate}
	\item [(i)] $|\varphi(u,t)|\le \varphi(e,\mathbbm{1}), ~~e\in\mathbb{H}\simeq\mathbb{R}^4, ~\mathbbm{1}\in \mathbb{R}^3$.
	\item[(ii)] $\varphi\left((u,t)^{-1}\right)=\overline{\varphi(u,t)};~x\in \mathbb{R}^4,~t\in\mathbb{R}^3.$
\end{enumerate}
Let $f:\mathbb{H}_q\longrightarrow \mathbb{C}$ be a $K$-invariant function on $\mathbb{H}_q$, i.e. $f(k\cdot(u,t))=f(k\cdot u,t)=f(u,t),~\forall~k\in K, (u,t)\in\mathbb{H}_q$, then the function $f(\cdot,t)$ defined by $$f(u,t)=f(|u|,t)=f\bigg(\frac{u}{|u|}\cdot |u|, t\bigg);$$ is a radial function on $\mathbb{H}_q$.
\section{Fourier Multipliers}
A Fourier multiplier is a linear operator that tends to alter the Fourier transform of a given function by multiplication. Some examples of multipliers are translations, differential operators, Hilbert transforms, etc.\\
Let $\mathfrak{F}_\lambda$ be the representation space of holomorphic functions such that $\disp{\|f\|_2=:\int|f|^2~d\mu <\infty,~~f \in L^2(\mathbb{H}_q)}$. The Fourier transform associated to an integrable and square integrable function $f$ on $\mathbb{H}_q$ is given by:
\begin{equation}\label{7.1}
	\hat{f}(\lambda):=\int_{\mathbb{H}_q} f(u,t)\pi_\lambda(u,t) du dt
	\end{equation}
	where $\lambda\in\Im \mathbb{H}\setminus\{0\},\lambda\ne 0,~(u,t)\in\mathbb{H}_q$ and $\pi_\lambda(u,t)$ is the irreducible unitary representation of $\mathbb{H}_q$ in the Fock's space $\mathfrak{F}_\lambda$ on $\mathbb{H}_q$ such that $\disp{J_\lambda\mapsto u\cdot \frac{\lambda}{|\lambda|}}$ defines a complex structure on $\mathbb{H}$.\\
	The action of the representation $\pi_\lambda(u,t)$ is described by;
	$$\pi_\lambda(u,t)f(q):=e^{i\langle \lambda,t\rangle}e^{-|\lambda|(|u|^2+2\langle q,u\rangle-2i\langle q\cdot\frac{\lambda}{|\lambda|},u\rangle}f(q+u).$$\\
	The Fourier Transform \eqref{7.1} acts on every function $f\in L^1(\mathbb{H}_q)$.\\
	Now suppose we have $\varphi, \psi\in\mathfrak{F}_\lambda$, a unique inner product defined on the Fock's space by the action of the Fourier transform is given as
	$$\langle\hat{f}(\lambda)\varphi,\psi\rangle=\int_{\mathbb{H}_q}f(u,t)\langle \pi_\lambda(u,t)\varphi, \psi\rangle dxdt, ~f\in L^2\mathbb{H}_q$$ such that $$|\langle\hat{f}(\lambda)\varphi, \psi\rangle|\le \|\varphi\|~\|\psi\|~\|f\|_1$$  implying $$|\langle \pi_\lambda(u,t)\varphi, \psi\rangle|\le \|\varphi\|~\|\psi\|;$$  and $\|\hat{f}(\lambda)\|_\infty\le\|f\|_1$ where $\|f\|_\infty:=\inf \{C\ge 0:|f(x)|\le C, \hbox{for ~almost~all~}x\}$. This clearly shows that $\hat{f}(\lambda)$ is a bounded operator on the Fock's space $\mathfrak{F}_\lambda$ and an Hilbert-Schmidt operator.
	
	\section{Spherical Fourier Transform}
	\begin{defx}
		Let $f$ be an integrable radial function on the quaternionic Heisenberg group $\mathbb{H}_q$ and $\mathcal{P}_n$, the vector space of all homogeneous polynomials of degree $n$. Let $\varphi\in \mathcal{P}_n$.\\
		Define $$\varphi_{\lambda,n}(u,t):=\frac{1}{n+1}e^{i\langle \lambda ,t\rangle}e^{-|\lambda|~|u|^2}\mathcal{L}^1_n(2|\lambda|~|u|^2);$$ as a bounded positive definite function on $\mathbb{H}_q$, where;
		$$\mathcal{L}^1_n(w)=e^{-w\slash 2}\left(\mathcal{L}^1_n(w) \big{\slash}\mathcal{L}_n^1(0)\right)$$ is the Laguerre function and $f(u,t)=F(|u|^2,t)$ a radial function, then the radial-spherical Fourier Transform is defined by,
		\begin{equation}
		\hat{f}(\lambda,n)=\int_{\mathbb{H}_q}f(u,t)\varphi_{\lambda, n}(u,t)du dt
		\end{equation}
	\end{defx}
This is the same as writing $$\hat{f}(\lambda, n)=\pi^2\int_{\mathbb{R}^3}\left( \int_0^\infty F(\alpha,t)\mathcal{L}_n^1(2|\lambda|\alpha~d\alpha)\right) e^{i\langle \lambda,t\rangle} dt$$
The following equalities are obvious.
\begin{enumerate}
	\item [i.] $\widehat{f*g}(\lambda,n)=\hat{f}(\lambda,n)\hat{g}(\lambda,n).$
	\item [ii.] $\widehat{L_{(u,t)}f}(\lambda,n)=\varphi_{\lambda,n}(-u,-t)\hat{f}(\lambda,n);~~~\hbox{where}~L_a f(w)=f(a^{-1}\cdot w)$.
	\item [iii.] $\widehat{R_{(u,t)}f}(\lambda,n)=\varphi_{\lambda,n}(u,t)\hat{f}(\lambda,n);~~~\hbox{where}~R_a f(w)=f(w\cdot a)$.
\end{enumerate}

\section{Main Result}
\subsection{Spherical Fourier Multipliers of Gelfand Pairs on $\mathbb{H}_q$}
	Recall the properties of the bounded Spherical function $\varphi_{\lambda, m}$ in the construction of the Spherical Fourier Transform. We now give the definition of Spherical Fourier Multiplier related to Gelfand pair on $\mathbb{H}_q$.\\
	Let $\mathfrak{m}:K\longrightarrow Aut(\mathbb{H}{q})$ be such that $k\longmapsto \mathfrak{m}_k$ is a morphism of any compact subgroup of automorphism $Aut(\mathbb{H}{q})$  of $\mathbb{H}{q}$, then the spherical Fourier Multiplier  related to Gelfand Pair on $\mathbb{H}_q$ is defined by:
	\begin{eqnarray*} T_{\mathfrak{m}}f(k,u,t)&:=&\int_K\mathfrak{m}(\varphi_{\lambda,n})\hat{f}(\varphi_{\lambda,n})\varphi_{\lambda,n}(k,u,t)~d\nu(\varphi_{\lambda,n}),~~~(k,u,t)\in\mathbb{H}_q\ltimes K\\
&=&\int_K\mathfrak{m}(\varphi_{\lambda,n})\int_{\mathbb{H}_q} f(u,t)\varphi_{\lambda,n}(u,t)\varphi_{\lambda,n}(k,u,t)~d\nu(u)d\nu(t)d\nu(k)\\
		&=&\int_K\int_{\mathbb{H}_q}\mathfrak{m}(\varphi_{\lambda,n}) f(u,t)\varphi_{\lambda,n}(k\cdot u,t)\varphi_{\lambda,n}(u,t)~d\nu(u)d\nu(t)d\nu(k).~~(by~Fubini~Tonelli's~theorem)
	\end{eqnarray*}
	Whence \begin{equation}\label{9.2}
	T_{\mathfrak{m}}f(k,u,t):=\int_{\mathbb{H}_q\ltimes K}\mathfrak{m}(\varphi_{\lambda,n}) f(u,t)\varphi_{\lambda,n}(k\cdot u,t)\varphi_{\lambda,n}(u,t)~d\nu(u)d\nu(t)d\nu(k)
	\end{equation}
Thus \eqref{9.2} defines a quaternionic Spherical Fourier Multiplier of Gelfand Pairs on $\mathbb{H}_q$.
	\begin{defx}
Given a bounded measurable function $\mathfrak{m}$  on $\mathbb{R}^n$ that acts on bounded positive spherical functions on $\mathbb{H}_q$, the transformation $T_{\mathfrak{m}}f(k,u,t)$ is well characterized by $$\widehat{T_mf}=\mathfrak{m}(\varphi_{\lambda,m})\hat{f}(\varphi_{\lambda,m});~~~f\in L^2(\mathbb{H}_q)$$
	$T_{\mathfrak{m}}$ is a bounded linear operator on $L^2(\mathbb{H}_q)$ by Plancherel Theorem.
	\end{defx}
\begin{rem}
The operator $T_{\mathfrak{m}}$ is called a Fourier multiplier for $L^p(\mathbb{R})$ if extended to $L^p(\mathbb{R})$.
\end{rem}
We observe that $T_{\mathfrak{m}}f(k,u,t)$ adequately satisfies the following results.
\begin{thm}
	Let $f,g,\in L^{1,\#}(\mathbb{H}_q)$, then the following equalities hold:
	\begin{enumerate}
		\item [i.] $T_{\mathfrak{m}}(f*g)=T_{\mathfrak{m}} f*g$.
		\item[ii.] $T_{{\mathfrak{m}}_1}f*T_{{\mathfrak{m}}_2}g=T_{{\mathfrak{m}}_1{\mathfrak{m}}_2}(f*g)$.
	\end{enumerate}
\end{thm}
\begin{proof}\ \\
\begin{enumerate}
		\item [i.] By \eqref{9.2}, we have;
\small \begin{eqnarray*}
	T_{\mathfrak{m}}(f*g)&=&\int_{\mathbb{H}_q\ltimes K}\mathfrak{m}(\varphi_{\lambda,n}) (f*g)(u,t)\varphi_{\lambda,n}(k\cdot u,t)\varphi_{\lambda,n}(u,t)~d\nu(u)d\nu(t)d\nu(k)\\
&=&	\int_{\mathbb{H}_q\ltimes K}\mathfrak{m}(\varphi_{\lambda,n}) \int_{\mathbb{H_q}}f(v,s)g\big[(-v,-s)\cdot(u,t)\big]\varphi_{\lambda,n}(k\cdot u,t)\varphi_{\lambda,n}(u,t)~d\nu(y)d\nu(s)d\nu(u)d\nu(t)d\nu(k)\\
&=&\int_{\mathbb{H}_q\ltimes K}\mathfrak{m}(\varphi_{\lambda,n}) f(v,s)\varphi_{\lambda,n}(k\cdot u,t)\varphi_{\lambda,n}(u,t)\int_{\mathbb{H_q}}g\big[(-v,-s)\cdot(u,t)\big]~d\nu(y)d\nu(s)d\nu(u)d\nu(t)d\nu(k)\\
&=&\int_{\mathbb{H_q}}\int_{\mathbb{H}_q\ltimes K}\mathfrak{m}(\varphi_{\lambda,n}) f(v,s)\varphi_{\lambda,n}(k\cdot u,t)\varphi_{\lambda,n}(u,t)d\nu(u)d\nu(t)d\nu(k)g\big[(-v,-s)\cdot(u,t)\big]~d\nu(v)d\nu(s)\\
&=&T_{\mathfrak{m}}f*g.
\end{eqnarray*}
		\item[ii.] By \eqref{9.2}, we have;
\small \begin{eqnarray*}
&&\int_{K\ltimes \mathbb{H}_q} \mathfrak{m_1}\hat{f}\varphi_1~d\nu(\varphi_1)*\int_{K\ltimes \mathbb{H}_q} \mathfrak{m_2}\hat{g}\varphi_2~d\nu(\varphi_2)\\
&=&\int_{\mathbb{H_q}}\bigg[\int_{\mathbb{H}_q\ltimes K}\mathfrak{m_1}\hat{f}\varphi_1~d\nu(\varphi_1)\int_{K\ltimes \mathbb{H}_q}\mathfrak{m_2}\hat{g}\varphi_2~d\nu(\varphi_2)\bigg]d\nu(v)d\nu(s)
\end{eqnarray*}
	\end{enumerate}
\end{proof}
Let $\mathfrak{m}:k\rightarrow Aut(\mathbb{H}_q)$ and $\mathfrak{m}\in L^1(\mathbb{H}_q)\cap L^2(\mathbb{H}_q)$ such that $\mathcal{F}^{-1} (m) \in L^{1,k}$ then, $$T_m: L^{2,k}(\mathbb{H}_q\ltimes K)\longrightarrow L^{2,k}(\mathbb{H}_q\ltimes K)$$ is bounded and $$\|T_m\|\le \|\mathcal{F}^{-1} (m)\|_{L^{1,k}(\mathbb{H}_q\ltimes K)}.$$
\begin{rem}
H\"{o}rmander gave a sufficient condition for $L^p$-multipliers which can be extended for Fourier multiplier for $L^p(\mathbb{H}_q)$.
\end{rem}
Now let $a> \frac{Q}{2}$, where $Q=4q+6$ is the degree of homogeneous dimension of $\mathbb{H}_q$ which plays the role of $n$ in the Euclidean settings and such that $\disp{\sup_{\nu >0}\|\nu^a \D^a \mathfrak{m}(\nu)\|\le C_a<\infty}$ and $\disp{(1+2^j|x|)^{-N}}$ is radial for all $N>0$. The condition $a> \frac{Q}{2}$ is sharp, i.e. the condition would fail if $a\le \frac{Q}{2}$. We thus state an equivalence of H\"{o}rmander theorem for $\mathbb{H}_q$.
\begin{thm}[Hormander Multiplier For $\mathbb{H}_q$] \ \\
Let $k=[\frac{n}{2}]+1$. If $\mathfrak{m}$ can be finitely differentiated ($k$-times) away from the origin and for any $\alpha\in \bf{N}^n$ which satisfies $|\alpha|\le k$, we have $$\sup_R R^{|\alpha|-\frac{n}{2}} \left(\int_{\mathbb{H}_q}|\D^\alpha \mathfrak{m}(\nu)|^2\chi_{R<|\nu|<2R}(\nu)d\nu\right)^{\frac{1}{2}}<\infty,$$ then $\mathfrak{m}$ is a Fourier multiplier for $L^p(\mathbb{H}_q)$ for $1<p<\infty$.
\end{thm}

\begin{proof}\ \\
Since $\disp{\D^\alpha m(R\cdot)(\nu)=R^{|\alpha|}(\D^\alpha m)(R\nu)}$ via the change of variable $\disp{\nu\mapsto R\nu}$ we obtain $\disp{\sup_R \left(\int_{\mathbb{H}_q}|\D^\alpha \mathfrak{m}(R\cdot)(\nu)|^2\chi_{1<|\nu|<2}(\nu)d\nu\right)^{\frac{1}{2}}\le C}$.\\
Now let $\psi \in C_c^\infty(\mathbb{R}^+)$ then $\disp{\D^\alpha (m(2^j\cdot)\psi)(\nu)=\sum_{|\beta|\le|\alpha|}C_{\beta, \alpha}\D^\alpha m(2^j\cdot)\D^{\alpha-\beta}\psi}$ and $\disp{|\D^\alpha \psi|\le C}$ hence $\disp{\sup_{\nu >0}\|\nu^a \D^a \mathfrak{m}(\nu)\| <\infty}$.\\
By Littlewood-Paley Decomposition, let $\psi \in C_c^\infty(\mathbb{R}^+)$ with support in $[\frac{1}{2}, 2]$ be such that \[
\sum_{j \in \mathbb{Z}} \psi(2^{-j} \lambda) = 1, \quad \forall \lambda > 0.
\]
We define \[
T_j f = \psi(2^{-j} \mathcal{L}) f.
\]
such that,
\[
T_m f = \sum_{j \in \mathbb{Z}} m(\mathcal{L}) T_j f;~~\text{where} ~\mathcal{L}~\text{is the sub-Laplacian defined by}~\mathcal{L} = - \sum_{j=1}^{4q} X_j^2.\]
where each $T_j$ is an integral operator with kernel $K_j$ and $m(\mathcal{L}) T_j$ has kernel $K_j^m$. By the smoothness condition of $m$, then rapid decay of the kernel is seen by the estimate \begin{equation}\label{7.2} |K_j^m(x)| \leq C 2^{jQ} (1 + 2^j |x|)^{-N}, \quad \forall N > 0. \end{equation}
\eqref{7.2} implies  $K^m$ satisfies Calder\'on--Zygmund kernel condition:
\[
|K^m(x)| \leq C |x|^{-Q}, \quad x \neq 0,
\]
\[
\int_{|x| > \varepsilon} K^m(x) \, dx = 0.
\]
Therefore, by Calder\'on--Zygmund theory for homogeneous groups, $T_m$ is bounded on $L^p(\mathbb{H}_q)$ for $1 < p < \infty$.\\
Summing over $j$ using square function estimates, we have;
\[
\|T_m f\|{L^p} \lesssim \left\| \left(\sum_j |T_j f|^2 \right)^{1/2} \right\|{L^p} \lesssim \|f\|_{L^p}.\]
\end{proof}

\bibliographystyle{amsplain}
\providecommand{\bysame}{\leavevmode\hbox to3em{\hrulefill}\thinspace}

\end{document}